\documentclass[11pt,reqno]{amsart}
\usepackage{amssymb,latexsym}
\usepackage{graphicx}
\usepackage{url}		
\usepackage[utf8]{inputenc}
\usepackage{mathtools}
\usepackage{todonotes}
\usepackage{tikz}
\tikzset{>=latex}
\usetikzlibrary{3d}
\usetikzlibrary{calc}
\usetikzlibrary{positioning}
\usetikzlibrary{petri}
\usetikzlibrary{arrows,backgrounds}
\usetikzlibrary{shapes}
\usetikzlibrary{graphs}
\usetikzlibrary{decorations.pathreplacing}
\usetikzlibrary{patterns}
\calclayout

\makeatletter
\newcommand{\journal}[1]{\def\@journal{\uppercase{#1}}}
\newcommand{\monthyear}[1]{%
  \def\@monthyear{\uppercase{#1}}}
\newcommand{\volnumber}[1]{%
  \def\@volnumber{\uppercase{#1}}}
\AtBeginDocument{%
\def\ps@plain{\ps@empty
  \def\@oddfoot{\@monthyear \hfil \thepage}%
  \def\@evenfoot{\thepage \hfil \@volnumber}}
\def\ps@firstpage{\ps@plain}
\def\ps@headings{\ps@empty
  \def\@evenhead{%
    \setTrue{runhead}%
    \def\thanks{\protect\thanks@warning}%
    \uppercase{\@journal}\hfil}%
  \def\@oddhead{%
    \setTrue{runhead}%
    \def\thanks{\protect\thanks@warning}%
    \hfill\uppercase{Sums of $k$-bonacci Numbers}}%
  \let\@mkboth\markboth
  \def\@evenfoot{%
    \thepage \hfil \@volnumber}%
  \def\@oddfoot{%
    \@monthyear \hfil \thepage}%
  }%
\footskip=25pt
\pagestyle{headings}%
}
\makeatother

\theoremstyle{plain}
\numberwithin{equation}{section}
\newtheorem{thm}{Theorem}[section]
\newtheorem{theorem}[thm]{Theorem}
\newtheorem{lemma}[thm]{Lemma}

\newtheorem{definition}[thm]{Definition}

\newtheorem{corollary}[thm]{Corollary}

\begin{document}
\monthyear{}
\journal{}
\volnumber{}

\setcounter{page}{1}

\title{Sums of  $k$-bonacci Numbers}
\author{Harold R. Parks}
\address {Department of Mathematics\\
 Oregon State University\\ 
 Corvallis\\ 
 Oregon 97331 USA}
\email{hal.parks@oregonstate.edu}
\author{Dean C. Wills}
\address{Cisco, San Francisco, California 94107 USA}
\email{dean@lifetime.oregonstate.edu}

\begin{abstract}

We give a combinatorial  proof of a formula giving
the partial sums of the $k$-bonacci sequence
as alternating sums of powers of two multiplied by binomial 
coefficients.
As  a corollary we obtain a  formula for the $k$-bonacci numbers.

\end{abstract}

\maketitle

\section{Introduction}

For $k \geq 1$, define $k$-bonacci numbers $f_n^{(k)}$ by the initial values and recursion\footnote{Note that the $k=2$ case gives the usual Fibonacci numbers
with the index shifted by $1$. The combinatorial approach to the
Fibonacci numbers in  \cite{10.4169/j.ctt6wpwjh}  employs this
shift. Extending that combinatorial approach to all  $k\geq 2$, we have
shifted onto the negative indices the beginning $k-1$ zeros that 
usually appear in the $k$-bonacci sequence. Details are given 
in Theorem \ref{combinatorial.shift}. 
Also, note that  $f_{n}^{(1)}=1$ for all $n \ge 0$.}
\begin{equation}\label{fib.recur}
f_n^{(k)} = 
\begin{cases}
       0, &   n < 0, \\
       1, &   n = 0, \\
            \sum_{i=1}^k f_{n-i}^{(k)}, & n\geq 1.
\end{cases}
\end{equation}
\vspace{1em}

The $k$-bonacci numbers were introduced under the name $k$-generalized Fibonacci numbers by  Miles   \cite{10.2307/2308649} in 1960.
According to \cite{Kessler:2004:CPG}, the Miles paper may well be the earliest paper 
on the topic to have appeared in a widely available journal.\footnote{The 
$3$-bonacci numbers  were mentioned by Agronomof  in a note that appeared in 1914 in \cite{Agronomof}.}
It appears that the terminology 
``$k$-bonacci'' was in use as early as 1973 by V. E. Hoggatt, Jr., and Marjorie Bicknell \cite{Bicknell:1973:GFP}. That terminology has continued to be used since then, though not universally.

This paper is about the following formula for the partial sums   of 
of  the sequence of $k$-bonacci numbers:
{\sl For $k\geq 1$ and $n\geq 0$, the 
sum of the first $n+1$ $k$-bonacci numbers
is given by\footnote{We  use  the notation 
$\lfloor \cdot \rfloor$ for the {\em floor function}, so $\lfloor x \rfloor$ equals the largest integer less than or equal to $x$.} }
\begin{equation}\label{main.comb}
\sum_{i=0}^{n} f_{i}^{(k)} = \sum_{j=0}^{\lfloor n/(k+1) \rfloor} (-1)^{j} \binom{n-jk}{j}2^{n-j(k+1)} 
\,.
\end{equation}

We obtained the formula (\ref{main.comb})  in the process of developing growth 
estimates for   sums of $k$-bonacci numbers
that we had hoped to apply to another problem. 
Since this formula seemed to be new,
we put aside the original problem to more fully understand the formula.
We found an algebraic proof for  (\ref{main.comb}), 
but we wondered if the formula could be understood 
at a deeper level---or at least be proved in another way---by  
using a  combinatorial approach.

Thanks to an anonymous referee 
(who used 
\cite{oeis_A172119} to track down the reference), we have learned 
that in fact
the formula  (\ref{main.comb}) should be credited to a 1925 paper by
Otto Dunkel.  
In Dunkel's paper  \cite{dunkel1925solutions}, if one simply compares
the equation for $P_2(n)$ in section 6 to the equation for $P_2(n)$ in section
10, then one sees  that  (\ref{main.comb})  holds.

While Dunkel's paper was concerned with various probability problems in coin tossing,
his method was to proceed via a difference equation and the roots of that
equation, so that 
at its heart Dunkel's  proof of (\ref{main.comb}) is algebraic.

In this paper, we give a  combinatorial proof of (\ref{main.comb}). Also,
because one can obtain the $k$-bonacci numbers from the differences of their  partial sums, 
we  obtain a corollary formula  (\ref{cor.eqn1}) 
for the $k$-bonacci numbers.
Our formula for the $k$-bonacci numbers is similar to 
equation (12) in Corollary 1 of \cite{Kessler:2004:CPG} and 
the formula in Theorem 2.4 of \cite{Howard:2011:SIF}.
It is  not related to the formulas for the $k$-bonacci numbers appearing in \cite{Lee:2001:BFR} and   \cite{10.2307/2308649}. It also does not occur in Dunkel's paper.

\section{Combinatorics}\label{combinatorics}

In \cite{10.4169/j.ctt6wpwjh}, Benjamin and Quinn begin with the observation
that the number of ordered sums of $1$s and $2$s that add to $n$ is the 
$(n+1)$st Fibonacci number. They then introduce a visual representation
that many, if not most, people will find easier to think about than sums
of $1$s and $2$s. They suggest thinking of the $1$s as squares and the $2$s as 
dominoes, so that  the number of ordered sums of $1$s and $2$s that add to $n$
is the number of ways to tile a ruler of length $n$ with squares and 
dominoes (see Figure \ref{fig.square.domino}).

When we generalize to the 
$k$-bonacci numbers
we need to consider the number of ordered sums
of $1$s, $2$s, $\ldots\ $, $k$s that add to $n$.
We will use the metaphor of covering a ruler of length $n$ by tiles of lengths
$1$ through $k$  (see Figure \ref{fig.ruler4}).
This metaphorical ruler will run from left to right.
The numbers obtained in this way are in fact 
$k$-bonacci numbers.

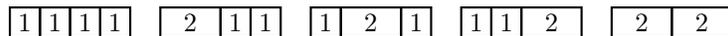
\begin{figure*}[htp]
\begin{center}
\begin{tikzpicture}[scale=0.8]

\draw[thick] (0.0,0) rectangle (0.5,.5) node[pos=.5] {\small 1};
\draw[thick] (0.5,0) rectangle (1.0,.5) node[pos=.5] {\small 1};
\draw[thick] (1.0,0) rectangle (1.5,.5) node[pos=.5] {\small 1};
\draw[thick] (1.5,0) rectangle (2.0,.5) node[pos=.5] {\small 1};

\draw[thick] (2.5,0) rectangle (3.5,.5) node[pos=.5] {\small 2};
\draw[thick] (3.5,0) rectangle (4.0,.5) node[pos=.5] {\small 1};
\draw[thick] (4.0,0) rectangle (4.5,.5) node[pos=.5] {\small 1};

\draw[thick] (5.0,0) rectangle (5.5,.5) node[pos=.5] {\small 1};
\draw[thick] (5.5,0) rectangle (6.5,.5) node[pos=.5] {\small 2};
\draw[thick] (6.5,0) rectangle (7.0,.5) node[pos=.5] {\small 1};

\draw[thick] (7.5,0) rectangle (8.0,.5) node[pos=.5] {\small 1};
\draw[thick] (8.0,0) rectangle (8.5,.5) node[pos=.5] {\small 1};
\draw[thick] (8.5,0) rectangle (9.5,.5) node[pos=.5] {\small 2};

\draw[thick] (10.0,0) rectangle (11.0,.5) node[pos=.5] {\small 2};
\draw[thick] (11.0,0) rectangle (12.0,.5) node[pos=.5] {\small 2};

\end{tikzpicture}
\end{center}
\caption{\label{fig.square.domino}The $f^{(2)}_4=5$ ways of tiling  a ruler of length 4 with squares and dominoes.}
\end{figure*}

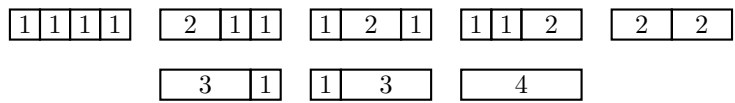
\begin{figure*}[htp]
\begin{center}
\begin{tikzpicture}[scale=0.8]

\draw[thick] (0.0,0) rectangle (0.5,.5) node[pos=.5] {\small 1};
\draw[thick] (0.5,0) rectangle (1.0,.5) node[pos=.5] {\small 1};
\draw[thick] (1.0,0) rectangle (1.5,.5) node[pos=.5] {\small 1};
\draw[thick] (1.5,0) rectangle (2.0,.5) node[pos=.5] {\small 1};

\draw[thick] (2.5,0) rectangle (3.5,.5) node[pos=.5] {\small 2};
\draw[thick] (3.5,0) rectangle (4.0,.5) node[pos=.5] {\small 1};
\draw[thick] (4.0,0) rectangle (4.5,.5) node[pos=.5] {\small 1};

\draw[thick] (5.0,0) rectangle (5.5,.5) node[pos=.5] {\small 1};
\draw[thick] (5.5,0) rectangle (6.5,.5) node[pos=.5] {\small 2};
\draw[thick] (6.5,0) rectangle (7.0,.5) node[pos=.5] {\small 1};

\draw[thick] (7.5,0) rectangle (8.0,.5) node[pos=.5] {\small 1};
\draw[thick] (8.0,0) rectangle (8.5,.5) node[pos=.5] {\small 1};
\draw[thick] (8.5,0) rectangle (9.5,.5) node[pos=.5] {\small 2};

\draw[thick] (10.0,0) rectangle (11.0,.5) node[pos=.5] {\small 2};
\draw[thick] (11.0,0) rectangle (12.0,.5) node[pos=.5] {\small 2};

\draw[thick] (2.5,-1.0) rectangle (4.0,-0.5) node[pos=.5] {\small 3};
\draw[thick] (4.0,-1.0) rectangle (4.5,-0.5) node[pos=.5] {\small 1};

\draw[thick] (5.0,-1.0) rectangle (5.5,-0.5) node[pos=.5] {\small 1};
\draw[thick] (5.5,-1.0) rectangle (7.0,-0.5) node[pos=.5] {\small 3};

\draw[thick] (7.5,-1.0) rectangle (9.5,-0.5) node[pos=.5] {\small 4};

\end{tikzpicture}
\end{center}
\caption{\label{fig.ruler4}The $f^{(4)}_4 = 8$ ways of tiling a ruler of length 4 with tiles of \mbox{lengths 1, 2, 3, 4.}}
\end{figure*}

\begin{theorem} \label{combinatorial.shift}
For $k\geq 1$ and $n\geq 0$, each $f_n^{(k)}$ equals the 
number of ways 
to cover a length $n$ ruler using tiles of lengths $1$ through $k$,
where by definition there are $0$ ways to cover a ruler of negative length.
\end{theorem}

\begin{proof}  Fix $k\geq 1$.  

Since there is exactly   one way to 
tile a ruler of length $0$ (and that is to use no tiles), and since 
there are $0$ ways to tile a ruler of negative length,
we see that the number of ways to tile 
a ruler of length $n$ using  tiles of lengths $1$ through $k$ satisfies the initial values of
(\ref{fib.recur}) for $n\leq 0$.

To complete the proof, we will show that for $n\geq 1$ the number of ways to tile
a ruler of length $n$ using  tiles of lengths $1$ through $k$ satisfies 
the recurrence relation of (\ref{fib.recur}). We argue inductively as follows:
For each tiling of the ruler of length $n$, there must be a rightmost tile.
That rightmost tile has a length $\ell\in\{1,2,\dots,k\}$, and in case
$n<k$ we also know that that rightmost tile has length $\ell \leq n$. 
The part of the
ruler to the left of the last tile can be tiled $f_{n-\ell}^{(k)}$ ways and we have
$f_{n-\ell}^{(k)} = 0$  in case $\ell > n$.
Thus it holds that 
\[
f_{n}^{(k)} = \sum_{\ell=1}^k f_{n-\ell}^{(k)}\,.
\]

\end{proof}



\begin{theorem}
For $k\geq 1$ and $n\geq 0$, we have
\begin{equation}\label{finally.named.it}
\sum_{i=0}^{n}	f_{i}^{(k)}= \sum_{i=0}^{\lfloor n/(k+1) \rfloor} (-1)^{i} \binom{n-ik}{i}2^{n-i(k+1)} 
\,.
\end{equation}
\end{theorem}

\begin{proof}  Observe that 
$ \sum_{i=0}^{n}	f_{i}^{(k)}$ is the number of ways to place tiles of lengths 
\mbox{$1$ through  $k$} end-to-end with total length not exceeding $n$.
To prove the theorem, we will count the number of ways that this can be done.

%
Let $U$ be the set of tilings of length not exceeding $n$
where there is no restriction on the sizes of the tiles that are used. 
The cardinality of $U$ is easily obtained as follows:
On a ruler of length $n$, place a hash mark at the $0$ position
and at each of the integer positions $1$ to $n$ either do or do not
place a hash mark.  Then between any two  hash marks place a tile that fills the space.
The last tile  ends at the last hash mark.  Because there are $2^n$ ways to place or not place 
the hash marks,  we have
\[
\#(\kern0.05em U\kern0.05em ) = 2^n
\,.
\]
In case $k\geq n$, no tile has length exceeding $k$, so
this value, $2^n$, equals the left-hand side of (\ref{finally.named.it}).
Also, when $k\geq n$,
the summation on the right-hand side of  (\ref{finally.named.it})  reduces to the
single term when $i=0$, that is, to $2^n$. Thus we see that (\ref{finally.named.it})  holds.
From here on we assume that $n>k$.

\medskip
To obtain the value on the  
left-hand side of (\ref{finally.named.it}) we must subtract from $2^n$ the
number of tilings  in $U$  that include at least
one tile of length greater than  $k$. To this end,
define $U_j$ to be the set of tilings in $U$ 
for which a  tile  that has  length greater than  $k$
has its right end at  integer position $j$. Note that this tile does 
not need to be the rightmost tile with length greater than $k$.
We need to evaluate
\[
\#\left( \bigcup_{1 \leq j\leq n} U_j \right)
\,.
\]
The Principle of Inclusion-Exclusion
(see for example  \cite{nelson2020brief} or \cite{stanley2012enumerative}) 
tells us that
\begin{equation}\label{value.to.subtract2}
\# \left( \bigcup_{1\leq j\leq n} U_j \right)
=
\sum_{1\leq j\leq n} \#  \Big(U_j\Big)
-
\sum_{1\leq j_1<j_2\leq n} \# \Big( U_{j_1}\cap U_{j_2} \Big) 
+
\cdots
\,.
\end{equation}

To evaluate the right-hand side of  (\ref{value.to.subtract2}),
we will need to  compute
\begin{equation}\label{sum.to.find}
\sum_{1\leq j_1<j_2<\cdots<j_i\leq n}
\#\Big( U_{j_1}\cap U_{j_2}\cap \cdots \cap U_{j_i} \Big)
\,.
\end{equation}
Note that the set  $U_{j_1}\cap U_{j_2}\cap \cdots \cap U_{j_i} $
will be empty unless 
\begin{equation}\label{nonempty.conditions}
k < j_1 \hbox{\rm\ \ and\ \ }
j_\ell+k  < j_{\ell+1}
\hbox{\rm \ \ for\ \ }\ell = 1,2,\dots,i-1
\,,
\end{equation}
because no two  tiles are allowed to overlap. Note that (\ref{nonempty.conditions})
tells us that $i(k+1)\leq n$, so only  when  $i\leq \lfloor n/(k+1) \rfloor$ will 
any of the summands in 
(\ref{sum.to.find}) be nonzero.

\medskip
Working with a ruler of length $n-ik$, we will show how to
evaluate (\ref{sum.to.find}). (The construction we are about to 
describe is illustrated in Figure \ref{fig.red.marks}.)
Choose $i$ numbers from $\{1,2,\dots,n-ik\}$.
Of course,
this can be done in $\binom{n-ik}{ i}$ ways. Label the
chosen numbers  $r_1,r_2,\dots,r_i$ so that   
\begin{equation}\label{conditions3}
0<  r_1 \hbox{\rm\ \ and\ \ }
r_\ell < r_{\ell+1}
\hbox{\rm \ \ for\ \ }\ell = 1,2,\dots,i-1
\,.
\end{equation}
For $\ell = 1, 2,\dots, i$, put a dashed hash mark at 
integer position $r_\ell$.

Next, place a normal hash mark at 0. 
There remain $n-i(k+1)$ unmarked positive integer positions between $1$ and $n-ik$.
At each of these remaining unmarked positions either do or do not place a normal hash mark.  This can be done in $2^{n-i(k+1)}$  ways.
Between any two  hash marks (dashed and dashed, normal and normal, or dashed and normal) place a tile that fills the space. 
The last tile  ends at the last hash mark, and the tiling has total length not exceeding $n-ik$.

\tikzset{
  shaded/.style = {gray, thin},
  redline/.style = {black, dashed, very thick},
  blackline/.style = {black, very thick},
  removed_redline/.style={black,  thin, dashed},
}

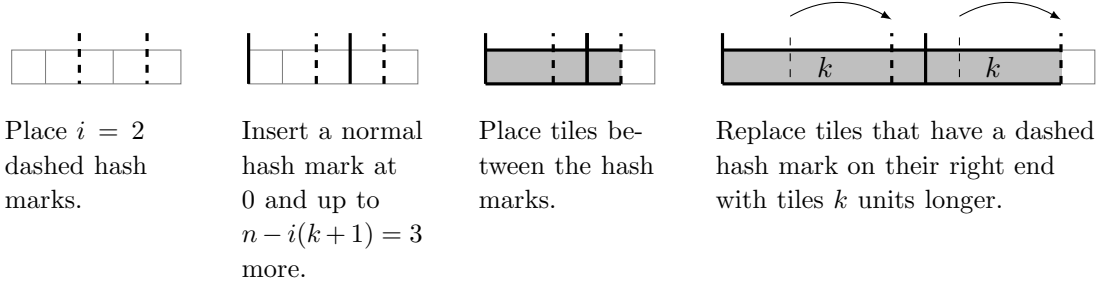
\begin{figure*}[htp]
\begin{center}
\begin{tikzpicture}[scale=0.9]

\draw[shaded] (-4.0,-0.5) rectangle (-1.5,-1.0) ;
\draw[shaded] (-3.5,-0.5) -- (-3.5,-1.0) ;
\draw[redline] (-3.0,-0.25)--(-3.0,-1.0) ;
\draw[shaded] (-2.5,-0.5) -- (-2.5,-1.0) ;
\draw[redline] (-2.0,-0.25)--(-2.0,-1.0) ;
\node[text width=2.4cm, anchor=north west] at ( -4.25 , -1.4 ) {\small Place $i=2$ dashed hash marks.};

\draw[shaded] (-0.5,-1.0) rectangle (2.0,-0.5) ;
\draw[shaded] (0.0,-1.0) -- (0.0,-0.5) ;
\draw[redline] (0.5,-1.0)--(0.5,-0.25) ;
\draw[blackline] (1.0,-1.0)--(1.0,-0.25) ;
\draw[redline] (1.5,-1.0)--(1.5,-0.25) ;
\draw[blackline] (-0.5,-1.0)--(-0.5,-0.25) ;
\draw[blackline] (4.5,-1.0)--(4.5,-0.25) ;
\node[text width=2.4cm, anchor=north west] at ( -0.75 , -1.4 ) {\small Insert a normal hash mark at 0 and up to $n-i(k+1) = 3$ more.};

\draw[shaded] (3.0,-1.0) rectangle (5.5,-0.5) ;
\filldraw[gray!50] (3.0, -1.0) rectangle (4.0, -0.5) ;
\filldraw[gray!50] (4.0, -1.0) rectangle (4.5, -0.5) ;
\filldraw[gray!50] (4.5, -1.0) rectangle (5.0, -0.5) ;
\draw[redline] (4.0,-1.0)--(4.0,-0.25) ;
\draw[blackline] (4.5,-1.0)--(4.5,-0.25) ;
\draw[redline] (5.0,-1.0)--(5.0,-0.25) ;
\draw[blackline] (3.0,-1.0)--(5.0,-1.0) ;
\draw[blackline] (3.0,-0.5)--(5.0,-0.5) ;
\draw[blackline] (3.0,-1.0)--(3.0,-0.25) ;
\draw[blackline] (4.5,-1.0)--(4.5,-0.25) ;

\node[text width=2.4cm, anchor=north west] at ( 2.75 , -1.4 ) {\small Place tiles between the hash marks.};

\draw[shaded] (11.5,-1.0) rectangle (12.0,-0.5) ;
\filldraw[gray!50] (6.5,-1.0) rectangle node[black, right] {$k$} (9.0,-0.5) ;
\filldraw[gray!50] (9.0,-1.0) rectangle (9.5,-0.5) ;
\filldraw[gray!50] (9.5,-1.0) rectangle node[black] {$k$} (11.5,-0.5) ;
\draw[removed_redline]  (7.5,-0.90)--(7.5,-0.20) ;
\draw[redline] (9.0,-1.0) -- (9.0,-0.25) ;
\draw[blackline] (9.5,-1.0) -- (9.5,-0.25) ;
\draw[blackline] (6.5,-1.0) -- (6.5,-0.25) ;
\draw[removed_redline]  (10.0,-0.90)--(10.0,-0.20) ;
\draw[redline]  (11.5,-1.0) -- (11.5,-0.25) ;
\draw[blackline] (6.5,-1.0) -- (11.5,-1.0) ;
\draw[blackline] (6.5,-0.5) -- (11.5,-0.5) ;

\draw [->,black] (7.5,-0.0) to [out=30,in=150] (9.0,-0.05);
\draw [->,black] (10.0,-0.0) to [out=30,in=150] (11.5,-0.05);

\node[text width=5.2cm, anchor=north west] at ( 6.25 , -1.4 ) {\small Replace tiles that have a dashed hash mark on their right end with tiles $k$ units longer.};

\end{tikzpicture}
\end{center}
\caption{\label{fig.red.marks}Example of the construction with $n=11,\ k=3,\ i=2$. One of three
possible normal hash marks is inserted.}
\end{figure*}

Finally,  each  tile that has its right end at a dashed hash mark is replaced with a tile that is $k$ units longer while the tiles to its right are moved along $k$ units to accommodate the longer tile. 
Thus we have created a tiling the length of which does not exceed $n$ and that for each $\ell = 1,2, \dots, i$ has a tile of  length greater than or equal to $k+1$ with its right end at  integer position $j_\ell := r_\ell + \ell k$. 
The $j_\ell$ satisfy (\ref{nonempty.conditions}) if and only if the $r_\ell$ satisfy condition (\ref{conditions3}).
Thus  we see that  
\[
\sum_{1\leq j_1<j_2<\cdots<j_i\leq n}
\#\Big( U_{j_1}\cap U_{j_2}\cap \cdots \cap U_{j_i} \Big)
=
\binom{n-i k }{ i}\, 2^{n-i(k+1)}\,.
\]
\end{proof}

\section{Corollaries}

\begin{corollary}
For $k\geq 1$, $n\geq 0$,
and  $m$ with ${\lfloor n/k \rfloor}  \geq m \geq {\lfloor n/(k+1) \rfloor}  $, it 
holds that
\begin{equation}
\label{eqn.comb2}
\sum_{i=0}^{n} f_i^{(k)}  
=\sum_{j=0}^{m} (-1)^{j} \binom{n-jk}{j}2^{n-j(k+1)} 
\,.
\end{equation}
\end{corollary}
\begin{proof}
The summands that are on the right-hand side of
(\ref{eqn.comb2}), but not on the right-hand side of 
(\ref{main.comb}), are those for which
$$
{\lfloor n/(k+1) \rfloor} <  j \leq {\lfloor n/k \rfloor} 
$$
holds. 
Note that the left-hand inequality is 
equivalent to $n-jk < j$ and the right-hand inequality 
is equivalent to $n-jk \geq 0$, so that for such a $j$ the binomial coefficient 
$\binom{n-jk}{j} $ equals $0$. 
\end{proof}

As a second  corollary we obtain a formula for the 
$k$-bonacci numbers
in terms of binomial coefficients and powers of $2$.
\begin{corollary}\label{main.cor}
For  $k\geq 1$ and $n \geq 0$, we have\footnote{To avoid  $0/0$  when
$n=j=0$, we use the 
Kronecker $\delta$ defined by $\delta_{i, j} = 1$
if $i=j$ and $\delta_{i, j} = 0$ if $i\neq j$.}
\begin{equation}\label{cor.eqn1}
f_{n}^{(k)}= 
\sum_{j=0}^{\lfloor n/(k+1)\rfloor} (-1)^j\,
\frac{(n-jk)+j+\delta_{n,0}}{2(n-jk)+\delta_{n,0}}\,  \binom{n-jk}{j} \,  2^{n-j(k+1)}
\,.
\end{equation}
\end{corollary}
\begin{proof} 
For $n= 0$, the summation on the right-hand side of  (\ref{cor.eqn1})
consist of only the $j=0$ term, so  the result is confirmed by (\ref{fib.recur}).

For  $k\geq 1$ and $n \geq 1$,
using (\ref{main.comb}), we have:
\[
f_{n}^{(k)}=
\sum_{i=0}^{n} f_i^{(k)} 
- \sum_{i=1}^{n-1} f_i^{(k)}
\]
and we will use    (\ref{eqn.comb2}) 
to replace the partial sums of  $k$-bonacci numbers in this equation.
We will  need to examine the upper limits in the summations that
occur  in  (\ref{eqn.comb2}).

Set $q = \lfloor n/(k+1)\rfloor $.  Then $n = q(k+1) + r$, where
$r$ is an integer with $0\leq r < k+1$.
We  see that 
$$
\frac{n-1}{k} = q + \frac{q+r-1}{k}
\,.
$$
Since $n\geq 1$, one or both of $q$ and $r$ must be positive.
We conclude that $\lfloor (n-1)/k \rfloor \geq q$, so
when    (\ref{eqn.comb2})  is applied with $n$ replaced by $n-1$
we may use $q$ as  the upper limit of summation.

We have
\begin{align}\nonumber
	f_{n}^{(k)} 
\nonumber
		&=\sum_{j=0}^{q} (-1)^{j} \binom{n-jk}{j}2^{n-j(k+1)} 
		- \sum_{j=0}^{q} (-1)^{j} \binom{(n-1)-jk}{j}2^{(n-1)-j(k+1)} \\
\nonumber 
		&=\sum_{j=0}^{q} (-1)^{j}  \left[ 2 \binom{n-jk}{j}- \binom{n-jk-1}{j} \right] \frac12\, 2^{n-j(k+1)} \\
\nonumber
		&=\sum_{j=0}^{q} (-1)^{j} 
		\frac12 \left[ \frac{2(n-jk)}{n-jk}  - \frac{n-jk-j}{n-jk} \right] \binom{n-jk}{j} 2^{n-j(k+1)} \\
\nonumber
		&=\sum_{j=0}^{q} (-1)^{j} 
	 \frac{(n-jk) + j}{2(n-jk)}  \binom{n-jk}{j} 2^{n-j(k+1)} 
\end{align}

\end{proof}




\bibliography{thebib,fibquart}

\def\cprime{$'$} \def\polhk#1{\setbox0=\hbox{#1}{\ooalign{\hidewidth
  \lower1.5ex\hbox{`}\hidewidth\crcr\unhbox0}}} \def\cprime{$'$}
  \def\cprime{$'$} \def\cprime{$'$} \def\cprime{$'$} \def\cprime{$'$}
  \def\cprime{$'$} \def\polhk#1{\setbox0=\hbox{#1}{\ooalign{\hidewidth
  \lower1.5ex\hbox{`}\hidewidth\crcr\unhbox0}}} \def\cprime{$'$}
  \def\cprime{$'$} \def\cprime{$'$} \def\cprime{$'$} \def\cprime{$'$}
  \def\cprime{$'$} \def\cprime{$'$} \def\cprime{$'$} \def\cprime{$'$}
  \def\cprime{$'$} \def\cprime{$'$} \def\cprime{$'$} \def\cprime{$'$}
  \def\cprime{$'$} \def\cprime{$'$} \def\cprime{$'$} \let \k = \c\ifx
  \undefined \arccot \def \arccot {{\rm arccot}} \fi\ifx \undefined \binom \def
  \binom {{\rm binom}} \fi\ifx \undefined \booktitle \def \booktitle #1{{{\em
  #1}}} \fi\ifx \undefined \cdprime \def \cdprime {$''$} \fi\ifx \undefined
  \cprime \def \cprime {$'$} \fi\ifx \undefined \lasp \def \lasp
  {\leavevmode\raise.45ex\hbox{$\lhook$}} \fi\ifx \undefined \mathbb \def
  \mathbb #1{{\bf #1}} \fi\ifx \undefined \mathcal \def \mathcal #1{{\cal #1}}
  \fi\ifx \undefined \sech \def \sech {{\rm sech}} \fi\ifx \undefined \url
  \input{path.sty} \fi
\begin{thebibliography}{10}

\bibitem{Agronomof}
M.~Agronomof.
\newblock Sur une suite r{\'e}currente.
\newblock {\em Mathesis: Recueil Math{\'e}matique}, 4:125--126, 1914.

\bibitem{10.4169/j.ctt6wpwjh}
Arthur~T. Benjamin and Jennifer~J. Quinn.
\newblock {\em Proofs that Really Count: The Art of Combinatorial Proof}.
\newblock Mathematical Association of America, 2003.

\bibitem{Bicknell:1973:GFP}
Marjorie Bicknell and V.~E. {Hoggatt, Jr.}
\newblock Generalized {Fibonacci} polynomials.
\newblock {\em Fibonacci Quarterly}, 11(5):457--465, December 1973.

\bibitem{dunkel1925solutions}
Otto Dunkel.
\newblock Solutions of a probability difference equation.
\newblock {\em The American Mathematical Monthly}, 32(7):354--370, 1925.

\bibitem{Howard:2011:SIF}
F.~T. Howard and Curtis Cooper.
\newblock Some identities for $r$-{Fibonacci} numbers.
\newblock {\em Fibonacci Quarterly}, 49(3):231--242, August 2011.

\bibitem{oeis_A172119}
The OEIS~Foundation Inc.
\newblock Sum the k preceding elements in the same column and add 1 every time,
  {E}ntry {A}172119 in the on-line encyclopedia of integer sequences, 2020.

\bibitem{Kessler:2004:CPG}
David Kessler and Jeremy Schiff.
\newblock A combinatoric proof and generalization of {Ferguson}'s formula for
  $k$-generalized {Fibonacci} numbers.
\newblock {\em Fibonacci Quarterly}, 42(3):266--273, August 2004.

\bibitem{Lee:2001:BFR}
Gwang-Yeon Lee, Sang-Gu Lee, Jin-Soo Kim, and Hang-Kyun Shin.
\newblock The {Binet} formula and representations of $k$-generalized
  {Fibonacci} numbers.
\newblock {\em Fibonacci Quarterly}, 39(2):158--164, May 2001.

\bibitem{10.2307/2308649}
Ernest~P. Miles, Jr.
\newblock Generalized {Fibonacci} numbers and associated matrices.
\newblock {\em The American Mathematical Monthly}, 67(8):745--752, 1960.

\bibitem{nelson2020brief}
Randolph Nelson.
\newblock {\em A Brief Journey in Discrete Mathematics}.
\newblock Springer International Publishing, 2020.

\bibitem{stanley2012enumerative}
Richard~P. Stanley.
\newblock {\em Enumerative Combinatorics}, volume~1.
\newblock Cambridge University Press, 2012.

\end{thebibliography}
\bibliographystyle{plain}

\medskip

\noindent MSC2020: 05A99, 11B39

\end{document}